\documentclass{amsart}

\usepackage{amssymb, amsmath, amsthm}
\DeclareMathOperator{\eldiag}{eldiag}
\newenvironment{ex}
{
 {\bf Example}
}
{
\vspace{.1in}
 }

\theoremstyle{definition}

\newtheorem{thm}{Theorem} 
\newtheorem{prop}{Proposition}[section] 
\newtheorem{llama}[prop]{Lemma} 
\newtheorem{lem}[prop]{Lemma}

\newtheorem{deff}{Definition}
\newtheorem{fact}[prop]{Fact}
\newtheorem*{claim}{Claim}
\newtheorem{notate}[prop]{Notation}
\newtheorem*{quest}{Question}
\newtheorem{obs}[prop]{Observation}
\newtheorem{cory}[prop]{Corollary}
\newtheorem{conje}{Conjecture}

\newcommand{\sthat}{\hspace{.1cm}| \hspace{.1cm}}
\newcommand{\C}{\mathcal{C}}

\newcommand{\acl}{\operatorname{acl}}
\newcommand{\im}{\operatorname{im}}
\newcommand{\zz}{\mathbb{Z}}

\newcommand{\dM}{\operatorname{dM}}
\newcommand{\RM}{\operatorname{RM}}

\newcommand{\Th}{\operatorname{Th}}
\newcommand{\donee}{$\Delta_1$-definitional expansion }
\newcommand{\donei}{$\Delta_1$-interdefinable }

\newcommand{\spec}{\operatorname{SRM}}
\newcommand{\Spec}{\spec}
\newcommand{\nbh}{\operatorname{Nbh}}
\newcommand{\Nbh}{\nbh}

\begin{document}
\title{Recursive Spectra of Strongly Minimal Theories Satisfying the Zilber Trichotomy}
\author{Uri Andrews and Alice Medvedev}
\thanks{The second author is partially supported by NSF FRG DMS-0854998 grant.}



\begin{abstract}
 We conjecture that for a strongly minimal theory $T$ in a finite signature satisfying the Zilber Trichotomy, there are only three possibilities for the recursive spectrum of $T$: all countable models of $T$ are recursively presentable; none of them are recursively presentable; or only the zero-dimensional model of $T$ is recursively presentable. We prove this conjecture for disintegrated (formerly, trivial) theories and for modular groups. The conjecture also holds via known results for fields. The conjecture remains open for finite covers of groups and fields. 
\end{abstract}

\maketitle

\section{Introduction}
\subsection{History and Definitions}
The following is a central question of recursive model theory.
\begin{quest}
Which (countable) models of a fixed theory $T$ are recursion-theoretically uncomplicated?
\end{quest}

 For us, ``uncomplicated'' means recursively presentable.

\begin{deff}
  A denumerable structure is \emph{recursive} if it has universe $\omega$ and the atomic diagram of the structure is recursive.\\
 A structure is \emph{recursively presentable} if it is isomorphic to a recursive structure.
\end{deff}

The \emph{spectrum problem}, first considered by Goncharov \cite{gonch} in 1978, asks which (isomorphism classes of) models of $T$ are recursively presentable; of course, such models must be countable. Similarly, our \textbf{languages are countable}.

Before tackling the spectrum problem for a theory $T$, one needs a complete list of countable models of $T$, indexed by something. \emph{Strongly minimal} theories are a natural test case, both because their models usually admit thorough, mathematically natural descriptions, and because countable models of any strongly minimal theory form an elementary chain $M_0\prec M_1 \prec M_2 \prec \ldots \prec M_{\omega}$, unless the theory is $\aleph_0$-categorical \cite{BaldwinLachlan}. This happens because each model is completely characterized by its algebraic dimension. It may happen that $T$ has no models of small dimension, so in general the dimension of $M_i$ is $d+i$, where $d$ is the dimension of the prime model of $T$. We say that $i$ is the \emph{dimension of $M_i$ over the prime model}. Now the spectrum problem for a strongly minimal theory $T$ comes down to understanding the \emph{spectrum} of $T$, defined to be
 $$\Spec(T)=\{i\sthat M_i \text{ is recursively presentable}\}$$

Since the spectrum of a strongly minimal theory is a subset of $\omega +1$, one can further ask about its recursion-theoretic complexity.
It is easy enough to come up with strongly minimal theories all of whose models are recursively presentable; and with theories that have no recursively presentable models at all. Goncharov \cite{gonch} first showed in 1978 that some strongly minimal theories have more interesting (than nothing or everything) spectra. It is possible to complicate the spectrum of a theory (even the theory of pure equality) by adding infinitely many new non-logical symbols and exploiting non-uniformity of their interpretations, all without changing the class of definable (with parameters) sets and, a fortiori, the underlying geometry.
To exclude such constructions, with the exception of the results in section \ref{cories}, this paper is exclusively about theories in \textbf{finite signatures}. Though the only known general bounds for the spectrum problem are non-arithmetical, the only known examples are computationally quite simple. Indeed, Herwig, Lempp, and Ziegler have to work quite hard in \cite{HLZ} to find an example of a strongly minimal theory in finite signature whose spectrum $\{0\}$ is neither everything nor nothing. In \cite{finlanguage} and \cite{0omega}, the first author showed that $\{0,\ldots, n\}$, $\omega$, and $\{\omega\}$ are spectra of theories in finite signatures by using a Hrushovski amalgamation construction. These constructions are well suited to building recursion-theoretically interesting strongly minimal structures, but, not surprisingly, these structures have little to do with the ones arising naturally in mathematics. Hrushovski \cite{Udi} invented the construction to provide a counterexample to Zilber's Trichotomy Conjecture. This conjecture, which is true for all strongly minimal theories that occur in nature, asserts that strongly minimal theories come in three flavors: \emph{disintegrated} (also known as ``trivial''); locally modular (also known as grouplike); and fieldlike. Though the precise definitions of ``fieldlike'' vary considerably in the literature, we take fieldlike to mean that the theory interprets a pure field, thus excluding Hrushovski fusions. Model theorists have a good understanding of the structure of models of strongly minimal theories satisfying the Zilber Trichotomy, leading us to conjecture that their spectra are quite simple.
For an introduction to the structure of models of strongly minimal theories, we refer the reader to \cite{marker}.

\subsection{Forecasts and Conjectures}

\begin{conje}\label{mainconj}
If $T$ is a strongly minimal theory in a finite signature satisfying the Zilber Trichotomy, then $\spec(T)$ is $\emptyset, \omega+1, \text{ or }\{0\}$.
\end{conje}

We prove this conjecture for disintegrated strongly minimal theories in Section 2, and for those locally modular strongly minimal theories which expand a group in Sections 3 and 4. The case of fieldlike fields is even easier: any (fieldlike) strongly minimal theory which expands a field is a definitional expansion \cite{Poizat}.
Such theories in finite signatures are decidable, so all of their models are recursively (even decidably) presentable \cite{leo}\cite{Khisamiev}.

\subsubsection{Finite covers}
Grouplike strongly minimal structures which are not groups arise via finite covers from groups, and similarly for fieldlike. These are the strongly minimal theories satisfying the Zilber trichotomy for which our conjecture still needs to be verified.

If $M$ is any non-disintegrated modular strongly minimal structure in finite signature, then by the group configuration theorem (\cite{Buechler}, p. 193 Theorem 4.5.2) $M$ interprets a modular strongly minimal group. Moreover, a careful reading of the proof of the group configuration theorem shows that there is a modular strongly minimal group which is $\Delta_1$-interpreted in $M$ (see definition below). This leaves open the question of spectra of finite covers of modular strongly minimal groups where the covering map is recursive.

A field-like strongly minimal theory is one which interprets a pure algebraically closed field. If in addition to interpreting an algebraically closed field, $M$ is also interpreted in an algebraically closed field, then we say that $M$ has a classical Zariski geometry. If $M$ has a classical Zariski geometry and $M$ has finite signature, then the theory of $M$ is recursive in the theory of the algebraically closed field, so is recursive. If $T$ is strongly minimal and recursive, then all models of $T$ are decidably presentable \cite{leo}\cite{Khisamiev}, so $\Spec(T)=\omega+1$. Not all field-like strongly minimal theories have classical Zariski geometries (\cite{HruZil}, Theorem C), and for such theories, the spectrum problem remains open.

\subsubsection{Spectrum $\{0\}$}
For disintegrated theories, all three possibilities in our conjecture are known to be possible, as the theory built by Herwig, Lempp, and Ziegler \cite{HLZ} is disintegrated.
It is easy to construct examples of modular strongly minimal groups with spectra $\omega+1$ and $\emptyset$, but we do not know whether there is a modular strongly minimal group in finite signature with spectrum $\{0\}$. It follows from our work that in any such example, the prime model must be of dimension 0, and the 
word problem for the (finitely generated) quasiendomorphism ring must not be solvable.
Macintyre \cite{Angus} constructed such division rings by embedding the word problem for a finitely generated group into the multiplicative group of a division ring, but his construction cannot be used directly for our purposes. In his division rings, the set of words representing the identity is $\Sigma_1$, while any recursive model gives a $\Pi_1$ criterion for words in the multiplicative group, making the set of multiplicative words representing the identity recursive. Indeed, the multiplicative word $w$ represents the identity if and only if $M \models (a,a)\in w$ for all $a$, which is $\Pi_1$ since the size of the image of an element under $w$ is uniformly recursive in the multiplicative word $w$.

\vspace{.2cm}

In the rest of this paper, we work with the following less transparent but stronger version of Conjecture \ref{mainconj}.
\begin{conje}\label{conjtoprove}
If $T$ is a strongly minimal theory in a finite signature satisfying the Zilber Trichotomy, and some positive dimensional $M \models T$ is recursive, then all models of $T$ are recursively presentable.
\end{conje}
Conjecture \ref{conjtoprove} implies Conjecture \ref{mainconj} since $\emptyset$, $\{0\}$, and $\omega+1$ are the only subsets of $\omega+1$ with the following property: if there is a $k\in S$ with $k>0$, then $S=\omega+1$.

 We prove Conjecture \ref{conjtoprove} for disintegrated strongly minimal theories at the end of section 2, and for strongly minimal modular groups at the end of the section 4. In both cases, we begin by passing to a different signature, in which the theory satisfies extra technical hypotheses. $\Delta_1$-interdefinability, described in the last part of this section, ensures that the recursion-theoretic properties are not affected by this change.

\subsection{Strongly minimal structures.}

Here we summarize some basic facts about strongly minimal structures;
proofs and details can be found in Sections 6.1-6.3 and 8.1-8.3 of
\cite{marker}.

 A theory $T$ is \emph{strongly minimal} if all of its models are
infinite, and every definable subset of every model of $T$ is finite
or cofinite. This property, satisfied by regular graphs of finite valence, by vector spaces, and by algebraically closed fields, has
surprisingly strong geometric consequences.

 The model-theoretic algebraic closure operator $\operatorname{acl}$
satisfies Steinitz exchange in models of strongly minimal theories,
making it a combinatorial pregeometry and giving rise to a notion of
independence for elements of the model and a notion of dimension,
called Morley rank, for definable sets and types.

 In algebraically closed fields, $\operatorname{acl}$ is the
field-theoretic algebraic closure, independence is algebraic
independence, and the dimension of a type is the transcendence degree
of a realization over the parameter set.
 In vector spaces, $\operatorname{acl}$ is linear span, independence
is linear independence, and the dimension of a type is the linear
dimension of a realization over the parameter set.

 A definable set $S$ of Morley rank $n$ has \emph{Morley degree $1$}
if it cannot be written as a union of two disjoint definable subsets
of Morley rank $n$. Every definable set of Morley rank $n$ can be
written as a finite union of $m$ definable sets of Morley rank $n$ and
Morley degree $1$; this $m$ is called the \emph{Morley degree} of $S$.
Courtesy of the French, we write $\operatorname{RM}(S)$ and
$\operatorname{dM}(S)$ for the Morley rank and degree of $S$.

 The Morley rank of a type is the minimum of Morley ranks of the
formulae occurring in it. 
Every definable set of Morley
rank $n$ and Morley degree $1$ has a unique extension to a complete
type of Morley rank $n$, called the \emph{generic type} of that
definable set; any realization of this type is a \emph{generic
realization} of the definable set.

 A strongly minimal structure $M$ is \emph{disintegrated} (formerly
known as trivial) if the associated pregeometry is trivial in the
sense that $\operatorname{acl}(A) = \bigcup \{ \operatorname{acl}(a)
\sthat a \in A \}$ for any set $A$. It follows that the relation $a \sim b$
defined by $a \in \left(\operatorname{acl}(b) \setminus
\operatorname{acl}(\emptyset)\right)$ is an equivalence relation on
non-algebraic singletons.

 Clearly, no group is disintegrated, since the relation $ab=c$ is not
essentially binary. When the lattice of algebraically closed subsets
of the strongly minimal structure is modular in the sense of lattice
theory, the strongly minimal set itself is called \emph{modular}. This
technical combinatorial property generalizes disintegration; it may be
enjoyed by a group $G$, in which case all definable subsets of $G^n$
for all $n$ are boolean combinations of subgroups definable without
parameters. See \cite{hrupilomo} or \cite[Fact 4.5]{BnB} for details. For groups, modularity
is equivalent to the more general local modularity.

\subsection{Strongly minimal modular groups}

Here, we collect some basic facts about strongly minimal modular groups; see \cite{marker} for a more thorough introduction.

Any strongly minimal group $G$ is abelian, so we write $+$ for the group operation. Further, $G_0 := \acl(\emptyset)$ is a subgroup of $G$, and the quotient $G/G_0$ is a vector space over a division ring $R$. This division ring is already present in the definable structure on $G$, and only depends on the theory of $G$. The elements of the division ring act on $G$ by finite-valued (as opposed to single-valued) group homomorphisms. More precisely, a \emph{quasiendomorphism} is an $\acl(\emptyset)$-definable rank $1$ subgroup $H$ of $G\times G$ which projects surjectively on the first coordinate\footnote{We note that others also require $H$ to be connected. We arrive at the same division ring by quotienting out by the constant quasiendomorphisms.}. The corresponding finite-valued function outputs $H(a) := \{ b \sthat (a,b) \in H \}$ on input $a$. The Lascar (in)equality for ranks then ensures that the $H(a)$ are generically finite.
It is easy to see that all $H(a)$ are cosets of $H(0)$, which then must be a finite $\acl(\emptyset)$-definable group, so $H(0) \subset \acl(\emptyset) = G_0$. Thus, $H(a)$ has the same finite size for all $a$. Thus, in the quotient $G/G_0$, the function associated to $H$, which is well-defined for similar reasons, is either a bijection or the 0 map. Multiplication is defined on the set of quasiendomorphisms by composition, and addition by $(x,y)\in \alpha+\beta$ if and only if there are $z_1$ and $z_2$ so that $(x,z_1)\in \alpha$, $(x,z_2)\in \beta$ and $z_1+z_2=y$. With this structure, the set of quasiendomorphisms almost forms a division ring over which $G/G_0$ is a vector space. ``Almost'' because a quasiendomorphism of the form $G \times F$ for a finite $F \leq G$ is non-invertible and acts on $G/G_0$ by sending every element to $0$. The ring $R$ is the quotient of the set of quasiendomorphisms by these \emph{constant} quasiendomorphisms. The basic idea of the proof of our conjecture for modular groups is to obtain recursive presentations of $G_0$ and $R$ from a recursive $G$ of positive dimension, and then to put them back together to obtain the other models of the theory of $G$.

 In preparation for doing that in section 4, we show in section 3 that it is sufficient to prove the conjecture under the additional hypothesis that $G$ is an \emph{abelian structure}: a commutative group $G=(G,\, +,\, -,\, 0,\, (H_i)_{i\in I}, \, (c_j)_{j\in J} )$ where each $n$-ary relation symbol $H_i$ in the signature defines a subgroup of $G^n$, and the $c_j$ are constant symbols. Notation and results from Blossier and Bouscaren \cite{BnB} are omnipresent in the discussion of abelian structures in section 4. One key fact that we use repeatedly is that in an abelian structure, every formula is equivalent to a boolean combination of positive primitive formulae. The \emph{positive primitive} (or just pp) formulae are those formed from atomic formulae by conjunction and existential quantification. It is not difficult to see that in an abelian structure, every pp-formula without parameters or constants defines a subgroup of $G^n$ for some $n$. As above, each abelian structure $G$ splits into two pieces: $G_0$ and $G/G_0$, each of which is itself a structure for the same signature as $G$. Following \cite{BnB}, we refer to the theory of $G/G_0$ as $T_1$; it is also strongly minimal. The definable structure inherited by $G/G_0$ is precisely that of a vector space over the quasiendomorphism ring $R$; in particular, the dimension of $G/G_0$ over $R$ determines its isomorphism type. Most importantly, the full definable structure on $G$ can be recovered as $G = G_0 \oplus G/G_0$, where we interpret $H_i$ as $H_i(G_0)\oplus H_i(G/G_0)$.

\subsection{Short outline and notation alerts}
We first discuss $\Delta_1$-interdefinability  and ``expansions by prime constants;'' we later repeatedly use these to recursively change signatures to more convenient ones. The second section is about the spectra of disintegrated theories. In the third section, we make the reduction from arbitrary modular strongly minimal groups to strongly minimal abelian structures. The fourth section extracts recursive presentations of the prime model and the quasiendomorphism ring from a recursive positive dimensional strongly minimal abelian structure. 

Formulae may have parameters, unless explicitly stated otherwise. A ``Signature'' is a set of nonlogical symbols; a ``language'' is a set of formulae. We identify an $n$-ary relation symbol $R$ in the signature with the atomic formula $R(x_1, \ldots, x_n)$.
\begin{deff}
 For an $n$-ary relation symbol $R$, we write $\RM(R)$ for the Morley rank of the formula $R(x_1, x_2, \ldots, x_n)$, and $\dM(R)$ for the Morley degree. When we write $\dM(R/ \emptyset)$, we mean the maximal number of disjoint subsets of $R$ of full Morley rank, \emph{definable over the empty set}. Similarly for \emph{generic over the empty set}.

 We say that a formula $\phi$ defining a group is connected if $\dM(\phi)=1$. We say that $\phi$ is connected over the empty set if $\dM(\phi/ \emptyset)=1$.
\end{deff}

\subsection{Changing signature without loss}

 Much of our work amounts to finding the ``right'' signature for a particular theory, so we often pass to definitional expansions and back to reducts. The recursion theory requires a refinement of these notions, as follows.

\begin{deff} \label{doneidef}
 We say that an $L_1$-theory $T_1$ and an $L_2$-theory $T_2$ are \emph{interdefinable} if some $L_1 \cup L_2$-theory is a definitional expansion of both.

 Given a strongly minimal $L$-theory $T$ and an $L$-formula $\phi$ without parameters, we say that $\phi$ is $\Delta_1$ (respectively $\Sigma_1$) if for any $M \models T$ the set $\phi^M$ defined by $\phi$ in $M$ is recursive (respectively recursively enumerable) in the atomic diagram of $M$. 

 For relational signatures $L \subset L'$, we say that an $L'$-theory $T'$ is a \emph{\donee} of $T'|_L$ if for every $R \in L' \setminus L$ the theory $T'$ proves that $R$ is equivalent to some $\Delta_1$ $L$-formula.

 If the signature of $T'$ is not relational, we pass to the obvious theory in the obvious relational signature interdefinable with $T'$ before applying the above definition.

 We say that an $L_1$-theory $T_1$ and an $L_2$-theory $T_2$ are \emph{\donei}if some $L_1 \cup L_2$-theory is a \donee of both.
\end{deff}

 For example, existential formulae are always $\Sigma_1$. So when $T$ is model-complete, every formula is $\Delta_1$. Thus, two model-complete theories are \donei if and only if they are interdefinable.


\begin{lem}
 The notion ``$T'$ is a \donee of $T$'' is transitive, and ``$T_1$ is \donei with $T_2$'' is an equivalence relation.
\end{lem}
\begin{proof}
Immediate from definitions.
\end{proof}

The point of these definitions and the reason we assume finite signatures throughout this paper is for the following proposition. Were $T_1$ and $T_2$ to have infinite signatures, we would need uniformity in the $\Delta_1$-interdefinability for the following proposition to hold.

\begin{prop}\label{doneisamespec}
 If $T_1$ and $T_2$ are \donei strongly minimal theories in finite signatures, then $\spec(T_1) = \spec(T_2)$.
\end{prop}

\begin{proof}
 It suffices to prove that $\spec(T_1) = \spec(T_3)$ whenever $T_3$ is a \donee of $T_1$. A recursive model of $T_1$ of dimension $n$ is obtained from a recursive model of $T_3$ of dimension $n$ by simply forgetting the extra information.
 A recursive model $M$ of $T_1$ of dimension $n$ produces a recursive model of $T_3$ of dimension $n$ because the interpretations of the finitely many new symbols are recursive subsets of the model, by definition of a $\Delta_1$ formula.
\end{proof}

One immediate consequence of Proposition \ref{doneisamespec} is that we may assume that our finite signatures are relational, which we will do throughout. Another form of recursive interdefinability between theories which we will need is that of adding or omitting constant symbols. As long as the constants which are added come from the prime model, this also does not change the spectrum.
\begin{deff}
 Let $T_0$ be a strongly minimal $L_0$-theory and let $T$ be an $L$-theory where $L := L_0 \cup \{ c_i \sthat i<n\}$. Then $T$ is called an \emph{expansion of $T_0$ by prime constants} if the $L_0$-type of $\bar{c}$ in $T$ is realized in the prime model of $T_0$.
\end{deff}
\begin{prop}\label{constantscanthurt}
 If $T$ is an expansion of $T_0$ by prime constants, then $\spec(T) = \spec(T_0)$.
\end{prop}
\begin{proof}
Given any recursive model of $T_0$ which is $n$-dimensional over the prime model, adding constant symbols to name elements in the prime model of $T_0$ does not change the dimension over the prime model. Further, adding interpretations for these finitely many new constant symbols preserves recursiveness of the model.
\end{proof}

 Notions like ``the Morley rank of an $L$-formula'' only make sense relative to a fixed (strongly minimal) $L$-theory $T$. We routinely abuse notation by not specifying the background theory, since we are always working with one particular theory. We further abuse notation by not even specifying the signature, which we are constantly varying, which is justified by the following easy observation.

\begin{obs}
 Suppose $T_1$ is a strongly minimal $L_1$-theory, $T_2$ is an $L_2$-theory which is interdefinable with $T_1$, and $\phi_2$ is the translation to $L_2$ of an $L_1$-formula $\phi_1$. Then $T_2$ is also strongly minimal, and the Morley rank and degree of $\phi_2$ in models of $T_2$ is equal to the Morley rank and degree of $\phi_1$ in models of $T_1$. The new theory $T_2$ also inherits further geometric properties of $T_1$, such as the Zilber Trichotomy classification.
\end{obs}

 One more piece of notation for changing signatures:

\begin{deff} \label{newsigdef}
 Given an $L$-theory $T$ and a set $S$ of $L$-formulae, we let $L(S) := \{ R_\phi \sthat \phi \in S\}$ be a new signature which has a relation symbol for each formula in $S$ and nothing else. We let $T(S)$ be the obvious $L(S)$ theory, that is the reduct to $L(S)$ of the $L \cup L(S)$-theory $T \cup \{ \forall x_1\, \ldots \forall x_{n_i}\, R_\phi(x_1, \ldots, x_{n_i}) \leftrightarrow \phi(x_1, \ldots, x_{n_i}) \}$.
\end{deff}

We continue to explicitly state the obvious.
\begin{lem} \label{newsiglem}
 In the above definition, if for every relational symbol $P$ in $L$ there is an $L(S)$-formula  $\psi_P$ such that $T \models (\forall x_1\, \ldots \forall x_n\, P(x_1, \ldots, x_n) \leftrightarrow \widetilde{\psi_P} (x_1, \ldots, x_n))$ where $\widetilde{\psi_P}$ is the $L$-formula obtained from $\psi_P$ by replacing each $R_{\phi}$ by $\phi$, then $T(S)$ is interdefinable with $T$. If in addition these $\psi_P$ and the $\phi \in S$ in the definition above are all $\Delta_1$, then $T$ and $T(S)$ are \donei.
\end{lem}

\section{Disintegrated theories}

 In this section we show that the spectrum of a disintegrated strongly minimal theory $T$ in a finite signature $L$ must be $\emptyset$, $\omega+1$, or $\{0\}$: nothing, everything, or just the prime model. It is clear that the first two are possible. An example of the last is the subject of \cite{HLZ}, from which this section borrows heavily. The key is to use the information in one recursive positive-dimensional model of $T$ to construct recursive models of $T$ of all dimensions. To do this, we pass to an new theory $\tilde{T}$ in a new signature $\tilde{L}$ interdefinable with $T$. This $\tilde{T}$ is then automatically \donei with $T$ because disintegrated strongly minimal theories are model-complete after naming a prime model \cite{GHLLM} (see Corollary \ref{triviallowrank}, p. \pageref{TLR}).

In what follows, we first prove the interdefinability of $T$ with a theory $\tilde{T}$ containing only relation symbols of rank 0 or 1. To do this, we inductively pass through many intermediate signatures $L'$ and theories $T'$, at each stage removing the symbol of highest Morley (rank, degree). The following lemma is the induction step of the construction.
	
 \begin{llama}
 Suppose that $T$ is a strongly minimal disintegrated theory in a signature $L$, and $R$ is a relation symbol in $L$ with $\RM(R)=k \geq 2$ and $\dM(R/\emptyset) = 1$.
Then there exists a finite set $S$ of $L$-formulae of Morley rank at most $(k-1)$ such that $T( S \cup L \setminus \{R\} )$ is interdefinable with $T$ in the sense of Definition \ref{newsigdef}.
\end{llama}

\begin{proof}
 Intuitively, all we need to show is that $R$ is definable in $L(S \cup L \setminus \{R\} )$; formally, we need $S\cup L\setminus \{R\}$ to satisfy the hypothesis of Lemma \ref{newsiglem}. The idea is to approximate $R$, up to pieces of lower Morley rank, by a product of definable sets of Morley rank $1$ and $0$; these definable sets and leftover pieces then constitute $S$.

 Take a generic over the empty set realization $(a_1, \ldots, a_n)$ of $R$ in some model $M$ of $T$.
 Since the theory is disintegrated, there is a partition $D \cup C_1 \cup C_2 \cup\ldots\cup C_k = \{ 1, \ldots, n\}$ with $a_i \in \acl(\emptyset)$ for $i \in D$, and $\RM((a_i)_{i \in C_j}) = 1$ for each $j \leq k$, and $\RM(a_i, a_{i'}) = 2$ whenever $i \in C_j \neq C_{j'} \ni i'$. For each $i \in D$, let the formula $\phi_i(x)$ be the projection of $R$ to the $i$th coordinate, and for each $j\leq k$, let $\psi_j(\bar{y})$ be the projection of $R$ onto the coordinates $\{x_i\sthat i\in C_j\}$. As $\dM(R/\emptyset) =1$, the Morley rank of each $\phi_i$ is 0 and the Morley rank of each $\psi_j$ is 1 and all $\phi_i$ and $\psi_j$ have Morley degree 1 over $\emptyset$.

 Now let $$ \alpha(x_1, \ldots, x_n) := \left( \bigwedge_{i \in D} \phi_i(x_i)\right) \wedge \left( \bigwedge_{j \leq k} \psi_j\left( (x_i)_{i \in C_j}\right)\right) .$$
Since $dM(\phi_i/\emptyset)=1$ and $dM(\psi_j/\emptyset)=1$, it follows that $\alpha$ and $R$ have the same Morley rank and degree over $\emptyset$. Now $\bar{a}$ is a generic realization of both, so  $\alpha \wedge R$ has the same Morley rank and degree over $\emptyset$ as $R$ and $\alpha$.

 Therefore, the two leftovers $\beta_1 := R \wedge \neg \alpha$ and $\beta_2 := \alpha \wedge \neg R$ have lower Morley rank. Clearly, $T$ proves that $R$ is equivalent to $(\beta_1 \vee (\alpha \wedge \neg \beta_2))$. Thus, $S := \{ \beta_1, \beta_2 \}\cup \{\phi_i \sthat i\in D\}\cup\{\psi_j \sthat j \leq k \}$ does the trick.
\end{proof}

In the next proposition, we apply this lemma to a relation symbol of maximal Morley rank as an induction step toward having only symbols of Morley rank 1 or 0 left in the end.

\begin{prop} \label{trivialreduces}
 Every strongly minimal disintegrated theory $T$ in a finite signature $L$ is interdefinable with some strongly minimal $\tilde{T}$ in a finite relational signature $\tilde{L}$ such that all relation symbols in $\tilde{L}$ have Morley rank 0 or 1.
\end{prop}

\begin{proof}
By replacing symbols in the signature by finitely many symbols with Morley degree 1 over $\emptyset$, we may assume all symbols of maximal Morley rank in the signature have degree 1 over the empty set.

 Let $k := \max \{ \RM(R) \sthat R \in L \}$, suppose $k>1$, and suppose that there are $m$ relation symbols in $L$ with this maximal Morley rank. We induct on the pair $(k, m)$ ordered lexicographically: at each induction step we decrease $k$ if $m=1$, and otherwise leave $k$ fixed and decrease $m$. The induction step is precisely the last lemma applied to one of the relation symbols of maximal Morley rank $k$. Proceeding in this manner yields the result.
\end{proof}

Having an interdefinable theory $\tilde{T}$, we now show that it is \donei with $T$. We use the following theorem from \cite{GHLLM}.
\begin{fact}[Goncharov, Harizanov, Laskowski, Lempp, McCoy \cite{GHLLM}]
Let $T$ be a disintegrated strongly minimal theory and let $M$ be a model of $T$. Then $\eldiag(M)$ is model complete.
\end{fact}
\begin{cory}\label{recintrivial}
Let $T$ be a disintegrated strongly minimal theory and let $M$ be a model of $T$. Then every definable set in $M$ is recursive in the atomic diagram of $M$.
\end{cory}
\begin{proof}
Let $S$ be a definable set in $M$. Then there exists $\bar{a},\bar{b}\in M$ and an existential formula $\phi(\bar{x},\bar{a})$ and a universal formula $\psi(\bar{x},\bar{b})$ so that $M\models S(\bar{x}) \leftrightarrow \phi(\bar{x},\bar{a}) \leftrightarrow \psi(\bar{x},\bar{b})$. Since $\phi$ gives a way to enumerate $S$ using the atomic diagram of $M$ and $\psi$ gives a way to enumerate $\neg S$ using the atomic diagram of $M$, we see that $S$ is recursive in the atomic diagram of $M$.
\end{proof}

\begin{cory} \label{triviallowrank}\label{TLR}
 Suppose that $T$ and $\tilde{T}$ are interdefinable disintegrated strongly minimal theories, then $T$ and $\tilde{T}$ are $\Delta_1$-interdefinable.
\end{cory}
\begin{proof}
Immediate from the previous Corollary and the definition of a $\Delta_1$ formula.
\end{proof}

Now by Proposition \ref{trivialreduces}, Proposition \ref{doneisamespec}, and Corollary \ref{triviallowrank}, it suffices to characterize spectra of strongly minimal disintegrated theories in finite signatures, under the additional assumption that each relation symbol has Morley rank 0 or 1.

\begin{prop} \label{trivialkey}
  Suppose that $T$ is a strongly minimal disintegrated theory in a finite relational signature $L$, and all relation symbols in $L$ have Morley rank 1 or 0. If $T$ has one recursive model of positive dimension, then all models of $T$ have recursive presentations.
\end{prop}

\begin{proof}
 Following \cite{HLZ}, we first define the \emph{connected component of $a$} for $a \in M \models T$ and prove that, for generic $a$, this coincides with the set of elements interalgebraic with $a$. We fix a recursive model $M$ of $T$ with positive dimension.

 Let $R$ be an $n$-ary relation symbol in $L$ and $i, j \leq n$, and suppose that the three projections $\pi_i(R) \subseteq M$, $\pi_j(R) \subseteq M$ and $\pi_{ij}(R) \subset M \times M$ all have Morley rank $1$. We define $(R,i,j)(a,b)$ to be the formula $\exists \bar{z}\, R(\bar{z})\wedge z_i=a\wedge z_j=b$. For each generic $a$, there are only finitely many elements $b$ such that $M\models (R,i,j)(a,b)\vee (R,i,j)(b,a)$. Thus, we can fix the finitely many $c_k$ $(k<l)$ so that $M\models \exists^{\infty}z\, (R,i,j)(c_k,z)$ as well as the finitely many $d_k$ $(k<m)$ so that $M\models \exists^{\infty}z\, (R,i,j)(z,d_j)$. We add constant symbols to the language to name each of these $c_k$ and $d_k$ for all $(R,i,j)$.

For each $(R, i,j)$, we define an edge relation $E_{Rij}$ by
 $$E_{Rij}(a,b):=\left(\left[\bigwedge_{j<l}a\neq c_j\right] \wedge \left[\bigwedge_{j<m} b\neq d_j\right] \wedge (R,i,j)(a,b)\right)\vee$$ $$\vee\left(\left[\bigvee_{j<l} a=c_j\right] \wedge \neg (R,i,j)(a,b)\right)\vee\left(\left[\bigvee_{j<m}b=d_j\right] \wedge \neg (R,i,j)(a,b)\right).$$ 
That is, $E_{Rij}(a,b)$ holds if and only if $(R,i,j)$ witnesses that $a$ and $b$ are interalgebraic. Since there are only finitely many $E_{Rij}$, and each defines a recursive set in each model by Corollary \ref{recintrivial}, adding interpretations of these symbols to a recursive presentation of a model of $T$ gives a recursive presentation of the expansion. Thus the expansion has the same recursive spectrum as $T$.

 We define the $n$th neighborhood of $a$, or $\nbh_n(a)$, for $a \in M \models T$ inductively for $n \in \omega$. Let $\nbh_0(a) := \{a\}$. Define $b \in \nbh_1(a)$ if there is some $E_{Rij}$ so that $E_{Rij}(a,b)$. Then $c \in \nbh_{n+1}(a)$ if there is some $b \in \nbh_n(a)$ such that $c \in \nbh_1(b)$. The connected component of $a$ is then $\nbh(a) := \bigcup_n \nbh_n(a)$.

\begin{claim}
Let $a\in M$ be generic. The set of elements interalgebraic with $a$ is precisely $\Nbh(a)$.
\end{claim}
\begin{proof} It is clear that ``$b \in \nbh(a)$'' is an equivalence relation (as $E_{Rij}(a,b)$ if and only if $E_{Rji}(b,a)$), and that $b \in \nbh(a)$ implies that $b \in \acl(a)$. To prove the converse, we take a saturated model $U$ of $T$ (dimension at least 2 suffices) and show that for any two independent $a$ and $b$ there is an isomorphism of $U$ fixing $U \smallsetminus (\nbh(a)\cup \nbh(b))$ and switching $a$ with $b$. If there were some $c$ interalgebraic with $a$ but not in $\nbh(a)$, it would now be forced into the algebraic closure of $b$, contradicting $a$ being independent from $b$. As both $a$ and $b$ are generic, $\nbh(a)$ is isomorphic to $\nbh(b)$ (as in \cite{HLZ}) over the finitely many elements named by constants (including the $c_k$ and $d_k$ from above). We fix an isomorphism $f$ of $\nbh(a)$ with $\nbh(b)$ and expand it to a map $g$ on $M$ by letting $g$ act via $f$ on $\nbh(a)$, letting $g$ act via $f^{-1}$ on $\nbh(b)$, and letting $g$ act via the identity on $M\smallsetminus (\nbh(a)\cup \nbh(b))$.

 We now show that $g$ is an isomorphism by considering tuples $\bar{c}$ on which some relation symbol $R$ holds. As each relation symbol has rank no greater than 1, no tuple to be considered contains elements of both $\nbh(a)$ as well as $\nbh(b)$. Relations holding on tuples entirely contained on one of $\nbh(a)$, $\nbh(b)$ or $M\smallsetminus (\nbh(a)\cup \nbh(b))$ are certainly preserved. By symmetry, we must only look at tuples $\bar{c}$ partially in $\nbh(a)$ and partially in $M\smallsetminus (\nbh(a)\cup \nbh(b))$. Each element of $\bar{c}\smallsetminus \nbh(a)$, by virtue of not being $E_{Rij}$ related to any element of $\bar{c}\cap \nbh(a)$ must be one of the finitely many $c_j$ or $d_j$ for the relation $(R,i,j)$. Thus, since the isomorphism of $\nbh(a)$ with $\nbh(b)$ was taken to be an isomorphism over the constants, this relation is preserved. Thus, the required automorphism exists, showing that interalgebraicity with a generic element $a$ is precisely the same as being in $\nbh(a)$.\end{proof}

 Let $\C$ be the isomorphism type over the constants of $\nbh(a)$ for a generic $a$. We have also shown that $b \in \acl(\emptyset)$ if and only if $\nbh(b) \not\cong \C$.

 From our recursive model $M$ containing a generic element $a$, we can recursively enumerate $\nbh(a)$. There are only finitely many edge-relations, and, aside from finitely many elements, each point has the same number of neighbors. Thus, we can recursively determine the isomorphism type of $\Nbh_i(b)$ for any element $b\in M$. This gives us an algorithm for enumerating $\acl(\emptyset)$: For each $i\in \omega$ and $b\in M$, recursively find $\Nbh_i(a)$ and $\Nbh_i(b)$, and check if they are isomorphic over the constants. If they are not, we see that $b$ is in $\acl(\emptyset)$. This makes both $\nbh(a)$ and $\acl(\emptyset)$ recursive $L$-structures over the constants.

Each model of $T$ is simply a union of one copy of $\acl_M(\emptyset)$ and some number of disjoint copies of $\nbh(a)$ over the constants. This gives a way to (uniformly) recursively construct each model of $T$.
\end{proof}

We have now proved the theorem we are after.

\begin{thm}\label{trivialspec}
If $T$ is a strongly minimal disintegrated theory in a finite signature, then $\Spec(T)=\emptyset, \{0\}, \text{ or } \omega+1$
\end{thm}

\begin{proof}
First apply Proposition \ref{trivialreduces} to obtain $\tilde{T}$ from $T$, which is $\Delta_1$-interdefinable by Corollary \ref{triviallowrank}. By Proposition \ref{doneisamespec}, $\spec(\tilde{T}) = \spec(T)$. By Proposition \ref{trivialkey}, if any positive dimensional model of $T$ is recursive, all models have recursive presentations. The only spectra consistent with this are $\emptyset, \{0\}, \text{ and } \omega+1$.
\end{proof}

\section{From modular group to abelian structure}
\subsection{The reduction}
 The purpose of this section is to take a strongly minimal modular group and to show that its spectrum is equal to the spectrum of a (related) strongly minimal abelian structure, so that what we prove in later sections about the spectra of abelian structures applies to strongly minimal modular groups in general.

 Let $L := \{+, R_1,\ldots, R_n, c_1,\ldots c_k \}$ and let $T$ be a strongly minimal modular $L$-theory implying that $+$ satisfies the group axioms. Our use of the symbol $+$ is justified since any strongly minimal group is abelian. In this section, we build a theory $T'$ of abelian structures in a language $L':=\{+,-,G_1,\ldots, G_l,c_1,\ldots c_{k'}\}$, which is \donei with $T$ in the sense of Definition \ref{doneidef}. Note that $-$ is $\Delta_1$-definable from $+$, since $x=-y$ if and only if $x+y=0$. In particular, we show the following:

\begin{thm}\label{thm-red-to-abstructs}
Let $T$ be the theory of a strongly minimal modular group in finite signature. Then there exists an interdefinable strongly minimal theory of abelian structures $T'$ such that the following conditions hold.
 \begin{itemize}
\item Each relation symbol in the signature of $T$ is equivalent to a boolean combination of cosets of groups in the signature of $T'$.
\item Each relation symbol in $T'$ is equivalent to a boolean combination of translates of relations in the signature of $T$.
\end{itemize}
\end{thm}
\begin{cory}
Let $T$ be the theory of a strongly minimal modular group in finite signature. Then there exists a $\Delta_1$-interdefinable theory $T'$ of abelian structures in a finite signature.
\end{cory}

 We struggle towards the abelian structure inductively, obtaining a new (finite) signature and a new theory, all \donei with the original $T$. At each step, the signature will be comprised of a set $\mathcal{G}$ of groups, a set $\mathcal{C}$ of constants, and a set $\mathcal{R}$ of other relation symbols. Our quest is to empty out $\mathcal{R}$ at the expense of growing the other two. Each induction step will decrease the number of relation symbols of highest Morley rank and degree in $\mathcal{R}$. Although the total number of symbols in $\mathcal{R}$ may increase, this induction is well-founded. The following proposition is the induction step, removing a maximal Morley rank and degree $R$ from $\mathcal{R}$ at the expense of adding new constant symbols to $\mathcal{C}$, a new group $G$ to $\mathcal{G}$, and two new relation symbols for $R \smallsetminus X$ and $X \smallsetminus R$ to $\mathcal{R}$, where $X$ is a defined coset of $G$.

\begin{prop}
 If $T$ is a strongly minimal modular theory in a language $L = \{ +, R, \ldots \}$ implying that $+$ defines a group, then after an expansion by prime constants, there is an $L$-$\Delta_1$-definable (in fact, a boolean combination of translates of $R$) group $G$ and a coset $X$ of $G$ such that $R \smallsetminus X$ and $X \smallsetminus R$ are of strictly lower Morley (rank, degree) than $R$. \end{prop}

The rest of this section constitutes the proof of this proposition. Since $T$ is modular, $R$ defines a finite boolean combination of cosets of $\acl(\emptyset)$-definable groups (Corollary. 4.8, \cite{BigPillay}). In the next lemma, we say the index of$\bar{b}_{ij}+H_{ij}$ in $\bar{b}_i+H_i$ for the index $[H_i\,:\,H_{ij}]$ noting that if $H_{ij}\leq H_i$ then any coset of $H_{ij}$ is completely contained in any coset of $H_i$ which it intersects.

\begin{lem}
There are $\acl(\emptyset)$-definable connected groups $H_i$ and $H_{ij}$ such that
 $$R = \bigcup_{i\leq m}\left((\bar{b}_i+H_i)\smallsetminus \bigcup_{j\leq k_i}(\bar{b}_{ij}+H_{ij})\right)$$
and $(\bar{b}_{ij}+H_{ij}) \leq (\bar{b}_i+H_i)$ have infinite index for all $i$, $j$.
\end{lem}

\begin{proof}
 Since the intersection of two cosets of groups is itself a coset of a group, this is just the disjunctive normal form. For the same reason, we may assume that each of the $\bar{b}_{ij} + H_{ij}$ are properly contained inside $\bar{b}_i+H_i$. If $H_i$ is not connected, we replace $(\bar{b}_i+H_i)\smallsetminus \bigcup_{j\leq k_i}(\bar{b}_{ij}+H_{ij})$ by the corresponding finite union of cosets; this ensures that $(\bar{b}_{ij}+H_{ij}) \leq (\bar{b}_i+H_i)$ cannot have finite index.
\end{proof}

Let $A$ be an $H_i$ of maximal Morley rank from the last lemma. This $A$ will be a finite index subgroup of the group $G$ in the proposition. In fact, $X$, the coset of $G$, will be the union of some of the cosets $\bar{b}_i + H_i$ where $H_i=A$. It is clear that the Morley (rank, degree) of $R\smallsetminus X$ and $X\smallsetminus R$ are each smaller than the Morley (rank, degree) of $R$.

It remains to produce a $\Delta_1$-definition of $G$. First we will take an intersection of translates of $R$ along $A$ to  remove the $H_{i'}$ where $H_{i'}\neq A$. Then we will take a union of translates of that intersection to fill in the missing pieces from the $H_{ij}$. Once done, we will have defined a union of cosets of $A$. This does not quite suffice, since we need a $\Delta_1$-definable group. After this, we will use the group-laws to strip away some of the remaining $A$-cosets until we have a single coset of a group.

To lighten notation, we reorder the expression in the last lemma to put all instances of $A$ at the beginning: for some $n \geq 1$, we have $H_i = A$ if and only if $i \leq n$.

\begin{notate}
For the remainder of the section, we will refer to the following:
 $$R = \bigcup_{i\leq m}\left(\bar{b}_i+H_i)\smallsetminus \bigcup_{j\leq k_i}(\bar{b}_{ij}+H_{ij})\right)$$
 $RM(H_i) \leq RM(A)$ for all $i$ \hspace{1cm} and \hspace{1cm} $H_i = A$ if and only if $i \leq n$
 $$B := \bigcup_{i\leq n}(\bar{b}_i+H_i) = \bigcup_{i\leq n}(\bar{b}_i+A)$$
\end{notate}

We first show that $B$ is quantifier-free definable from $R$ and a few new constants by showing that $B$ is a finite union of translates of a finite intersection of translates of $R$.

\begin{lem}
 (a) There are finitely many $\bar{a}_\alpha\in A$ such that
 $S := \bigcap_{\alpha}(\bar{a}_\alpha+R) \subseteq B$.

 (b) For any such $\bar{a}_\alpha\in A$, there are finitely many $\bar{e}_\beta\in A$ such that
 $B = \bigcup_{\beta} (\bar{e}_\beta+S)$

\end{lem}

\begin{proof}
(a)
 We intersect $R$ with its translates, trying to trying to remove portions of $R$ which are not in $B$. Rather than working with $R$, we work with $R' := \bigcup_{i\leq m}(\bar{b}_i+H_i) = B \cup (\bigcup_{i > n}(\bar{b}_i+H_i))$; since $R \subseteq R'$, it clearly suffices to show that $\bigcap_{\alpha}(\bar{a}_\alpha+R') \subseteq B$. We let $\bar{a}_0 := 0$.
 Since $A$ is connected, $H_i\cap A$ has infinite index in $A$ for $i>n$. Thus, we can find $\bar{a}_\alpha \in A$ for $1 \leq \alpha \leq m+1$ (actually, $(m-n+1)$ suffices) such that $\bar{a}_\alpha - \bar{a}_{\alpha'} \notin H_i$ for all $i > n$ and all $\alpha \neq \alpha'$. The purpose of taking $\bar{a}_\alpha \in A$ is to ensure $\bar{a}_\alpha+B=B$. The purpose of taking
 $\bar{a}_\alpha - \bar{a}_{\alpha'} \notin H_i$ is to ensure that $(\bar{a}_\alpha + H_i) \cap (\bar{a}_{\alpha'} + H_i) = \emptyset$. We will show that $\bigcap_{\alpha}(\bar{a}_\alpha+R') \subseteq B$. The idea is that by translating generically enough along $A$, each of the intersections of the $H_i$ for $i>n$ will be empty. We claim that now $\bigcap_{\alpha}(\bar{a}_\alpha+R') \subseteq B$:

$$ \bigcap_{\alpha}(\bar{a}_\alpha+R')  \setminus B =
 \bigcap_{\alpha} \left( \left((\bar{a}_\alpha+B) \cup \bigcup_{i>n} (\bar{a}_\alpha + \bar{b}_i + H_i)\right)   \setminus B \right)=$$
$$= \bigcap_{\alpha} \left( \left( \bigcup_{i>n} (\bar{a}_\alpha + \bar{b}_i + H_i)\right) \setminus B \right) \subseteq
 \bigcap_{\alpha} \left( \bigcup_{i>n} (\bar{a}_\alpha + \bar{b}_i + H_i)\right)$$

Rewriting the last line as a union of intersections of $(m+1)$ terms of the form $(\bar{a}_\alpha + \bar{b}_i + H_i)$, we see that in each intersection, some $\bar{b}_i+ H_i$ occurs more than once, with different $\bar{a}_\alpha$'s, rendering the intersection empty. The first claim of the lemma is now proved.

(b)
For the second claim, let
 $$B_0 : = \bigcup_{i\leq n}\left((\bar{b}_i+H_i)\smallsetminus \bigcup_{j\leq k_i}(\bar{b}_{ij}+H_{ij})\right)$$
As $B_0\subseteq R$, we have that $S_0 :=\bigcap_{\alpha}(\bar{a}_\alpha+B_0) \subseteq  \bigcap_{\alpha}(\bar{a}_\alpha+R)=S\subseteq B$. For each $\alpha$, $B \setminus (\bar{a}_\alpha+B_0)=\bigcup_{j\leq k_i}(\bar{a}_\alpha+\bar{b}_{ij}+H_{ij})$ is a finite union of cosets of proper subgroups of $A$. So $B \setminus S_0$ is contained in $\bigcup_{j\leq k_i,\alpha}(\bar{a}_\alpha+\bar{b}_{ij}+H_{ij})$. In particular, $S$ is generic in $B$ in the sense of stable group theory (\cite{BigPillay}, Lemma 6.12), so finitely many translates of it cover $B$. 
\end{proof}

To complete the proof of the proposition, we need to obtain the desired finite-index supergroup $G$ of $A$ which is $\Delta_1$-definable in terms of $B$, a finite union of cosets of $A$.

 We take a translate $B - \bar{b}_1$ of $B$ that contains $A$ and successively remove cosets of $A$ until we get a group. In the end, we get a group $G$ such that $A \leq G \subseteq B - \bar{b}_1$, and its translate $X := G+\bar{b}_1$ satisfies the proposition.

 Here is how we begin the induction: Let $G_0 := B - \bar{b}_1$ be a translate of $B$ that contains $A$; let $\bar{c}_i : = \bar{b}_i - \bar{b}_1$ and let $I_0 := \{ 1, \ldots, n \}$ so that $G_0 = \bigcup_{i \in I_0} (A + \bar{c}_i)$.

 Here is the inductive assumption: $B - \bar{b}_1 = G_0 \supseteq G_1 \supseteq G_2 \ldots \supseteq A$ are finite unions of cosets of $A$, and $G_{r+1} \subsetneq G_r$ unless $G_r$ is a group.

 And here is the induction step: let $I_{r} \subset \{ 1, \ldots, n \}$ be such that $G_{r} = \bigcup_{i \in I_{r}} (A + \bar{c}_i)$, and let $G_{r+1} := \bigcap_{i \in I_r} G_r - c_i$. 

 It is clear that $G_{r+1}$ is still a union of cosets of $A$. If $G_{r+1}=G_r$, then $G_{r}$ is a finite union of cosets of $A$ which is closed under addition, thus a group.

All of the parameters used here may be chosen from a prime model, so by an expansion by prime constants, we have that $G$ is $\Delta_1$-definable. We are now done proving the proposition.

\subsection{Corollaries of the reduction}\label{cories}

Here we collect some recursive model theoretic corollaries of Theorem \ref{thm-red-to-abstructs}. Along with the fact that in any theory of abelian structures every formula is equivalent to a boolean combination of pp-formulae, an explicit axiomatization of the theory of abelian structures is given in Blossier and Bouscaren (\cite{BnB}, Cor. 2.4). We obtain the following three corollaries which hold \textbf{even for theories with infinite signatures}. These can be seen as analogs of similar results about disintegrated theories in \cite{GHLLM}. Recall that almost model complete theories are those where every formula is equivalent to a boolean combination of existential formulae.

\begin{cory}
Every modular strongly minimal group is almost model complete after naming constants for a model.
\end{cory}
\begin{proof}
Given a formula $\phi$, we can restrict to the subsignature of $L$ consisting of only the non-logical symbols occurring in the formula $\phi$. Our analysis gives a quantifier-free translation of $T\vert_L$ to an abelian structure using the constants from a model. In this abelian structure, every formula is equivalent to a boolean combination of existential formulae, and every atomic formula in the abelian structure is quantifier-free definable in $T$. Thus $\phi$ is equivalent to a boolean combination of existential formulae in $T$.
\end{proof}

\begin{cory}\label{Cory-Two-Jumps}
Let $T$ be the theory of a strongly minimal modular group with a recursive model. Then $T\leq_T 0''$, that is $T$ is Turing reducible to the set of true $\exists\forall$-sentences in $(\mathbb{N},+,\cdot)$.
\end{cory}
\begin{proof}
Given a sentence $\phi$, we need to show that $0''$ can determine whether $\phi\in T$. We first restrict to the reduct of $T$ to the finite signature comprised of symbols occurring in $\phi$. Using $0''$, we can identify some sequence of $\Delta_1$-formulae $\psi_i$ which define groups such that the original relations are boolean combinations of cosets of the $\psi_i$. In fact $0'$ suffices for this step, as we will use in Corollary \ref{Cory-One-Jump}. Thus, $0''$ need only be able to enumerate the axiom list provided by Blossier and Bouscaren (\cite{BnB}, Cor. 2.4) in this new language. We include that list here:
\begin{itemize}
\item (Abelian Groups) The axioms of commutative groups
\item (Abelian Structure) For each $H$ in the signature, the axiom that states that $H$ is a subgroup of $M^{\text{arity}(H)}$.
\item (Equivalence Sentences) Sentences of the form $\forall \bar{x}\left(\phi(\bar{x})\leftrightarrow \psi(\bar{x})\right)$ where $\phi$ and $\psi$ are pp-forumulae which define the same group.
\item (Dimension Sentences) For each pair $H\subseteq H'$ of pp-definable subgroups of $M$, such that the index of $H$ in $H'$ is equal to $n$, the sentence ``$[H':H]=n$''. For each pair $H\subseteq H'$ of pp-definable subgroups of $M$ such that the index of $H$ in $H'$ is infinite, the infinite scheme of sentences ``$[H':H]\geq k$'', for every $k\geq 1$. 
\item (Constants) For each pp-definable $H$, $H$-congruences and non-$H$-congruences between tuples of constants.
\end{itemize} 

Given two pp-formulae, it is $\forall\exists$ using the recursive model $M$ to declare that they are equivalent, thus $0''$ can compute which equivalence sentences are true. Similarly, upon finding $k$ tuples  $x_1,\ldots, x_k$ from $M$ such that $\bigcup_{i\leq k} (x_i + H)=H'$, $0''$ can enumerate the axiom $[H':H]=k$. Similarly, upon finding $k$ tuples $x_1,\ldots, x_k$ from $M$ such that $\bigcup_{i\leq k} (x_i + H) \subseteq H'$, $0''$ can enumerate the axiom $[H':H]\geq k$. This takes care of the dimension sentences, and congruences between constants are existentially defined, thus $0''$ can determine these as well.
\end{proof}

In fact, a more careful analysis yields a stronger result in the case that the recursive model is of positive dimension. This is surprising, since this result does not hold for disintegrated theories. In fact, there exists a disintegrated strongly minimal theory $T$ all of whose models admit recursive presentations, yet $T\equiv_T 0''$.

\begin{cory}\label{Cory-One-Jump}
Let $T$ be a strongly minimal modular group with a recursive model of positive dimension. Then $T\leq_T 0'$. 
\end{cory} 
\begin{proof}
We fix a recursive model $M\models T$ of positive dimension and an element $a\in M\smallsetminus \acl(\emptyset)$. Given a sentence $\phi$, we again must determine whether $\phi \in T$. We again restrict to the subsignature generated by the non-logical symbols occurring in $\phi$. Since the groups in the interdefinable abelian structure are in fact boolean combinations of translates of quantifier-free definable sets in $T$, $0'$ can identify these, thus performing the translation to an abelian structure. Again, $0'$ needs to be able to enumerate the axioms from (\cite{BnB}, Cor. 2.4). This is done by performing an analysis of a given pp-definable group as follows.

\begin{claim}[1]
For any pp-definable group $H$, $0'$ can (uniformly) determine the rank of $H$. 
\end{claim} 
\begin{proof}
It suffices to show that $0'$ can determine whether or not the rank of $H$ is $\geq k$. The rank of $H$ is $\geq k$ if and only if the projection of $H$ onto some $k$ coordinates is $M^k$. This is equivalent to the projection of $H$ onto some $k$ coordinates containing each of the $k$ elements $(a,0,\ldots, 0)$ through $(0,\ldots, 0, a)$. The projection is an existential formula, so $0'$ can determine whether this is true. 
\end{proof}
\begin{claim}[2]\label{equalgroups}
For any two pp-definable groups $H$ and $H'$, $0'$ can determine if $H=H'$.
\end{claim}
\begin{proof}
Firstly, $0'$ verifies that the rank of $H$ equals the rank of $H'$ and that there are $k$ coordinates such that both $H$ and $H'$ project onto $M^k$ via those coordinates. This is done as above. If this is so, we let $\pi$ be the projection map onto these fixed $k$ coordinates. $0'$ then determines the size of $\pi^{-1}(0,0,\ldots, 0)\cap H$ and $\pi^{-1}(0,0,\ldots, 0) \cap H'$. If these are not equal, then certainly $H\neq H'$. If they are equal, then $0'$ determines whether they are equal as sets. Finally, $0'$ determines whether $\pi^{-1}(a,0,\ldots, 0)\cap H = \pi^{-1}(a,0,\ldots,0)\cap H'$, and similarly for each permutation of $(a,0,\ldots, 0)$. Knowing the size of these sets, $0'$ then verifies whether they are equal. If these are equal, then we show $H=H'$. Since any element can be written as the sum of two generics, we have that $\pi^{-1}(x,0,\ldots,0)\cap H = \pi^{-1}(x,0,\ldots, 0)\cap H'$ for any $x$. Similarly for any permutation of $(x,0,\ldots, 0)$. Finally, as we show below, $H$ and $H'$ are the sum of all these pre-images, so $H=H'$.
\end{proof}

Using this claim, $0'$ can now enumerate the axiom list exactly as in Corollary \ref{Cory-Two-Jumps}.
\end{proof}

\section{Strongly minimal abelian structures}

 In this section, we continue defining new signatures and new theories, taking care not to change the spectrum. The first two reductions, to a language where all relation symbols define strongly minimal subgroups of $M^n$, are $\Delta_1$-interdefinable reductions. We then, as in the case of disintegrated theories, use the trace of our signature on a binary language to determine algebraicity. First, by analyzing algebraic formulae, we will verify that algebraic closure is the same in this new language. Then we will use the binary language to get a natural presentation of the quasiendomorphism ring of the structure. We will do this by first showing that we can recursively present the ring of quasiendomorphisms generated by those explicitly in the language, and then showing that every quasiendomorphism is represented by one of these.

  Now that we are inside an abelian structure, we have two nice quantifier elimination results.

\begin{fact}
 \begin{enumerate}
\item Every formula is equivalent to a Boolean combination of pp-formulae. (\cite[Fact 2.3]{BnB})
\item Every definable (with parameters) connected group is pp-definable over $\emptyset$. (\cite[Proposition 2.6]{BnB})
\end{enumerate}\end{fact}

\subsection{First reductions}

\begin{lem} \label{wlogcon}
 For every strongly minimal theory $T$ of abelian structures in a finite signature $L$, there is a strongly minimal theory $\tilde{T}$ of abelian structures in a finite signature $\tilde{L}$ such that \begin{itemize}
 \item $T$ and $\tilde{T}$ have the same recursive spectrum.
 \item $\tilde{T}$ proves that each relation symbol in $\tilde{L}$ is a \emph{connected} group.
 \item $L$ and $\tilde{L}$ have the same number of relation symbols of each Morley rank.
\end{itemize}
\end{lem}

\begin{proof}
 For each relation symbol $G$ in $L$, let $G_0$ be the connected component of $G$. Remember that $G_0$ is a finite-index subgroup of $G$. We first take an expansion $T_2$, in signature $L_2$, of $T$ by finitely many constants $\{ c_{G,i} \sthat G \in L, i \leq [G : G_0] \}$ from the prime model, to represent each coset of $G_0$ in $G$; then Proposition \ref{constantscanthurt} guarantees that $\spec(T_2) = \spec(T)$. Then we note that by the second quantifier elimination result above, each $G_0$ is defined by an existential formula. In $T_2$, each coset of $G_0$ in $G$ contains an interpretation of a constant symbol, so each of these cosets is also defined by an existential formula. Now the quantifier-free definable (recursive in the atomic diagram of a model) $G$ is a finite union of existentially definable (recursively enumerable in the atomic diagram of a model) cosets of $G_0$, forcing each coset to be definable by both a universal and an existential $L_2$-formula. Let $\phi_G$ be an $L_2$ formula defining $G_0$. As $T_2$ proves that $\phi_G$ is equivalent to both an existential and a universal formula, $\phi_G$ is $\Delta_1$. In $\tilde{L}$, we shall have names for all these constants and for all the connected $G_0$ but not for any $G$. Formally, we apply Definition \ref{newsigdef} to the set of $L_2$-formulae $S := \{+\}\cup \{ \phi_G \sthat G \in L \} \cup \{ c_{G,i} \sthat G \in L, i \leq [G : G_0] \}$ and note that the hypotheses of Lemma \ref{newsiglem} are satisfied, so $\tilde{T} := T_2(S)$ is \donei with $T_2$, and thus has the same spectrum. The second and third conclusions follow because ``connected'' and ``Morley rank'' did not change from $T$ to $\tilde{T}$.
\end{proof}

\begin{lem} \label{wlogrankone}
 Suppose that $T$ is a strongly minimal theory of abelian structures in a finite signature $L$, and all relation symbols in $L$ are connected. Then there is a strongly minimal theory $\tilde{T}$ of abelian structures in a finite signature $\tilde{L}$ such that \begin{itemize}
 \item $T$ and $\tilde{T}$ have the same recursive spectrum.
 \item all relation symbols in $\tilde{L}$ have Morley rank $1$.
\end{itemize}
\end{lem}

\begin{proof}
 As in the last lemma, the proof involves an expansion by finitely many constants from the prime model, and a change of signature as in Definition \ref{newsigdef}.
 Remove all relations symbols in $L$ that have Morley rank 0. By connectedness, these can only define the element $0\in M^n$ for some $n$.

 For each $n_G$-ary relation symbol $G$ in $L$, let $r_G := RM(G) \geq 1$. In what follows, we make reference to a generic element of $G$, which is where we make use of the connectedness of $G$.
 To lighten notation, assume that the first $r_G$ coordinates of a generic element of $G$ are its basis, that is to say that the projection onto the first $r_G$ coordinates is a surjective group homomorphism $\pi_G : G \rightarrow M^{r_G}$ with finite kernel $K_G$. We first expand  to a new signature $L_2$ by a set $C$ of constants naming (all coordinates of) all elements of all the finite kernels $K_G$.

 For each $G$ and for each $i \leq r_G$, let $G_i$ be the kernel of the projection of $G$ to the first $r_G$ coordinates excluding $i$, defined by
  $$ \phi_{G,i} := (x_1, \ldots, x_{r_G}, y_{r_G+1}, \ldots, y_{n_G}) \in G \wedge (\bigwedge_{j \leq r_G, j\neq i} x_j=0)$$
Note that $RM(G_i) = 1$, as the projection to the $i$th coordinate is surjective, with finite kernel $K_G$.
We shall apply Definition \ref{newsigdef} to $L_2$ with $S := \{ \phi_{G,i} \sthat G \in L, i \leq r_G \}\cup \{\text{constant symbols in }L_2\}$.
Each formula in $S$ is quantifier-free definable in $L_2$.
To finish proving the Lemma, it suffices to show that each symbol in $L_2$ is $\Delta_1$-definable in $L_2(S)$. That is, we need to produce a $\Delta_1$-definition of $G$ from the constants naming elements in $K_G$ and from the $G_i$.

 \begin{claim}
  The group $G$ can be recovered from the $G_i$ and $K_G$ as follows: $G = G_1 + G_2 + \ldots + G_{r_G} + K_G$.
 \end{claim}
 \begin{proof}The inclusion  $G \supseteq G_1 + G_2 + \ldots + G_{r_G} + K_G$ is obvious, since all $G_i$ and $K_G$ are subgroups of $G$. On the other hand, since $\pi_G$ is surjective, for each $(a_1,\ldots, a_{r_G}, \bar{b}) \in G$, there are elements $(0, \ldots, a_i, \ldots, 0; \bar{c_i}) \in G_i$ for each $i$. For any choice of these, their sum $(\bar{a}, \sum_i \bar{c_i})$ is in $G$ and differs from $(\bar{a}, \bar{b})$ by something in the kernel of $\pi_G$, which is $K_G$. \end{proof}

 Thus, given an atomic diagram of a model of $T_2(S)$, to check whether $(\bar{a}, \bar{b})$ is in $G$ we need only find some tuple $(0, \ldots, a_i, \ldots, 0; \bar{c_i})$ for each $i\leq r_G$, and check whether $(\bar{0}, \bar{b}-\sum_i \bar{c_i})$ is in $K_G$, which is a recursive procedure. We have now verified all the hypotheses of Lemma \ref{newsiglem}, and its conclusion yields that $\Spec(\tilde{T})=\Spec(T)$.
\end{proof}

 After applying Lemma \ref{wlogcon} one more time, we may assume that all the group symbols in our language are connected and rank one, i.e. strongly minimal. By replacing groups of the form $H\times \{0\}$ by $H$, we may also assume that each of these strongly minimal subgroups of $M^n$ projects surjectively onto all coordinates.

 Finally, we make one last easy reduction before commencing our analysis of algebraicity. The theorem which we are aiming to prove states that given any recursive positive dimensional model of $T$, we can recursively present every countable model of $T$. We can now pass to the theory $T'$ we get by removing all constant symbols from the signature (even if these constants no longer are represented in the prime model). This is because given any recursive positive dimensional model $M$ of $T$, removing the constants yields a recursive positive dimensional model of $T'$, from which the theorem gives a recursive presentation of every countable model of $T'$. Since every model of $T$ is a model of $T'$ with finitely many constants named, each countable model of $T$ also has a recursive presentation.

\subsection{First step towards a binary language.}\label{firstbinary}
We have now reduced to the case of strongly minimal abelian structures with no constants where all relations have Morley rank and degree 1. To prove Conjecture \ref{conjtoprove}, we fix a recursive positive dimensional model $M$ of such a theory $T$ and endeavor to show that all models of $T$ are recursively presentable. 

We now introduce a finite collection of subgroups of $M^2$ which we use to analyze the structure $M$. It is easy enough to describe the new signature $\tilde{L}$, and easy enough to see that all symbols in $\tilde{L}$ are $\Delta_1$-definable in $L$.

 For each $n_G$-ary relation symbol $G \in L$ and each $i < n_G$, let $G_i \leq M \times M$ be the projection of $G$ to the $i$th and $(i+1)$st coordinates. Clearly, these are all strongly minimal. Since there is a fixed $n\in \omega$ such that for each $x\in M$ there are exactly $n$ elements $y$ in $M$ such that $(x,y)\in G_i$, the $G_i$ are $\Delta_1$-definable in $L$. Thus, taking $S := L \cup \{ G_i \sthat G \in L, i < n_G \}$, it is clear that $T^+ := T(S)$ (in the sense of Definition \ref{newsigdef} and Lemma \ref{newsiglem}) is \donei with $T$. Passing from this $T^+$ to the reduct $\tilde{T}$ without the original high-arity symbols for the groups $G$ is not so easy. For one thing, $T^+$ need not even be a definitional extension of $\tilde{T}$. It may be the case that only a finite-index supergroup of $G$ is definable in $\tilde{T}$. The next Lemma shows that a finite index supergroup of $G$ is definable in $\tilde{T}$, and the following example shows that $G$ may not be definable in $\tilde{T}$.

\begin{lem} \label{finindsuperlemma}
 For every $G \in L$, some finite-index supergroup of $G$ is quantifier-free definable in $\tilde{L}$ from the last paragraph.
\end{lem}
\begin{proof}
 For an $n$-ary relation symbol $G$ in $L$, let $\psi_G(x_1, \ldots, x_n)$ be the $\tilde{L}$-formula $ \bigwedge_{i < n} G_i(x_i, x_{i+1})$. The group $\hat{G}$ defined by $\psi_G$ is the fiber product of the $G_i$. Clearly, $G \leq \hat{G}$ and, since $\hat{G}$ is still Morley rank $1$ (because, for example, the kernel of the projection of $\hat{G}$ to the first coordinate is a fiber product of finite groups, so itself finite), $G$ must have finite index in $\hat{G}$.
\end{proof}

 \begin{ex} We build a structure with universe $3^\mathbb{Z}$ for the signature $\{ +, t, e\}$. We interpret $+$ to be the usual group operation on the direct product of countably many copies of $\mathbb{Z}/3 \mathbb{Z}$. We interpret the binary $t$ as the quasiendomorphism where $t(\vec{a},\vec{b})$ if and only if there is some $c \in \mathbb{Z}/3 \mathbb{Z}$ such that  $b_{i+1} := a_i +c$ for each $i$. Note that this $t$ is a three-to-three function (correspondence, if you prefer). To verify that every rational function in the field $\mathbb{Z}/3 \mathbb{Z}(t)$ is a quasiendomorphism, note that the kernel of any polynomial in $\mathbb{Z}/3 \mathbb{Z}[t]$ is finite (for example, the kernel of $t^3-1$ is the 3-periodic elements of $3^\mathbb{Z}$), consisting of (not necessarily all) sufficiently periodic elements of $3^\mathbb{Z}$. On the other hand, all the countably many periodic elements are in the algebraic closure of the empty set. We define the ternary $e$ to be a finite-index subgroup of $t(x,y) \wedge t(y,z)$, without defining any new quasiendomorphisms: $e(\vec{a}, \vec{b}, \vec{d})$ holds if and only if $b_{i+1} = a_i+c$ and $d_{i+1} = b_i +c$, for the \emph{same} $c$. Note that projecting $e$ onto any two coordinates just gives back $t$ or $t^2$, while taking fibers is useless for creating new quasi-endomorphisms as they are all finite.
 \end{ex}

 We will, in fact, utterly give up on recursively reconstituting a model of $T(S)$ out of a given recursive model of $\tilde{T}$. Here is what we do instead (labeled by the subsection in which the step appears):
 \begin{description}
 \item[\ref{firstbinary}] Given one recursive positive-dimensional model $M$ of a strongly minimal theory $T$ of abelian structures in a finite signature $L$ with all relation symbols in $L$ strongly minimal, expand it to a recursive positive-dimensional model $M^+$ of $T^+$ above, and take the reduct to a recursive positive-dimensional model $\tilde{M}$ of $\tilde{T}$ above. Note all three structures have the same universe $U$.
 \item[\ref{algebraicformulae}] Prove that the closure operators on subsets of $U$ given by algebraic closure in $M^+$ and in $\tilde{M}$ are the same. Conclude that they have the same quasiendomorphism ring, and that the algebraic closure of the empty set in $\tilde{M}$ and the algebraic closure of the empty set in $M^+$ are the same subset $U_0$ of $U$.
 \item[\ref{kashrut}] Extract from $\tilde{M}$ (a positive-dimensional recursive model of a strongly minimal theory $\tilde{T}$ of abelian structures where all relation symbols are binary and strongly minimal) a recursive presentation of its quasiendomorphism ring $R$, and a recursive enumeration of its algebraic closure of the empty set, $U_0$. Conclude that the prime model of $T^+$ is now recursive, since $U_0$ is recursively enumerable.
 \item[\ref{triumph}] Cite proposition 2.13 from \cite{BnB} which says that any model of $T^+$ is a direct sum of its prime model with a vector space over its quasiendomorphism ring $R$, and use this to give a recursive presentation of any countable model of $T^+$.
 \end{description}

The first item was achieved using Definition \ref{newsigdef} and Lemma \ref{newsiglem} in the beginning of this subsection.

\subsection{Algebraic formulae via matrices}\label{algebraicformulae}

\textbf{Setup recap:} $M$ is a positive-dimensional recursive model of a strongly minimal theory $T$ of abelian structures in a finite signature $L$ with all relation symbols in $L$ strongly minimal. $M^+$ and its theory $T^+$ are the $\Delta_1$-definitional expansions of $M$ and $T$ to a new language $L^+$ which has relation symbols $R_{G_i}$ for each $n_G$-ary relation symbol $G \in L$ and each $i < n_G$, with $R_{G_i}$ interpreted in ${M^+}$ to be the (strongly minimal) projection $G_i$ of $G$ to the $i$th and $(i+1)$st coordinates. $\tilde{M}$ and its theory $\tilde{T}$ are the reducts of $M^+$ and $T^+$ to a signature $\tilde{L}$ which has all the new relation symbols $R_{G_i}$, but none of the original high-arity $G$ from $L$. All three structures $M$, $M^+$, and $\tilde{M}$ share the same universe $U$. Of the three algebraic closure operators, $\acl^M$ and $\acl^{M^+}$ are obviously identical (definitional expansion) and will be denoted by $\acl$; and the third $\acl^{\tilde{M}}$ will be denoted $\tilde{\acl}$.

The quest of this subsection is to show that $\tilde{\acl} = \acl$. More precisely,
\begin{prop}
 For any $\bar{a}, b \in U$, if $b \in \acl(\bar{a})$, then there exists an $\tilde{L}$-formula $\theta(\bar{x},y)$ such that $\{ b' \sthat \tilde{M} \models \theta(\bar{a},b') \}$ is finite and contains $b$.
\end{prop}

\begin{proof}
We begin with the $L$-formula witnessing algebraicity and convert it into an $\tilde{L}$-formula.
By corollary 2.5 in \cite{BnB}, we know that for $\bar{a}, b \in U$ with $b \in \acl(\bar{a})$, there is a pp formula $\alpha(\bar{x},y) := \exists \bar{z}\, \bigwedge_{k\leq N} \alpha_k(\bar{x},y,\bar{z})$ such that \begin{itemize}
 \item Each $\alpha_k$ is of the form
 $$ \alpha_k(\bar{x},y,\bar{z}) := (\bar{d}_{k1}  \cdot \bar{x} + e_{k1}  \cdot y + \bar{p}_{k1}  \cdot
 \bar{z}\,,\, \ldots \, ,\,\bar{d}_{kl}  \cdot \bar{x} + e_{kl}  \cdot y + \bar{p}_{kl}  \cdot
 \bar{z}) \in H_k$$
where $\bar{d}_{ki}$, $e_{ki}$, $\bar{p}_{ki}$ are tuples of integers, and $H_k$ is a strongly minimal group projecting surjectively onto each coordinate, one of finitely many relation symbols in $L$; for actual equations, we allow $H_k = \{ 0 \}$.
 \item The set defined by $\alpha(\bar{a},y)$ in $M$ is finite.
\end{itemize}

We now collect $\bigwedge_k \alpha_k(\bar{x},y,\bar{z})$ into a matrix equation with a collection of rows for each $k$. Let $A_k$ be the matrix with $i$th row $(\bar{d}_{ki}, e_{ki})$; let $P_k$ be the matrix with $i$th row $\bar{p}_{ki}$. From these, form $A$ to be the matrix with the $A_k$ vertically stacked; and form $P$ to be the matrix with the $P_k$ vertically stacked. Let $K: = \Pi_k H_k$. Then
  $$\alpha(\bar{x},y) := (\exists \bar{z}\, A \cdot (\bar{x}, y) + P \cdot \bar{z} \in K ) \leftrightarrow$$
replace $z$ by $-z$
 $$ \leftrightarrow (\exists \bar{z}\, A \cdot (\bar{x}, y)  \in K + P\cdot \bar{z} ) \leftrightarrow$$
$$ \leftrightarrow ( A \cdot (\bar{x}, y)  \in K + I_P )$$
where $I_P$ is the range of the $\zz$-linear function defined by the matrix $P$, i.e. $w \in I_P$ if and only if $\exists \bar{z}\, P \cdot \bar{z} = w$. So now
\begin{equation}\label{reducedformalg} \alpha(\bar{x},y) \leftrightarrow ( (\bar{x}, y)  \in A^{-1} (K + I_P) )\end{equation}
where $A^{-1}$ is a harmless preimage under a linear function, with no assumptions on invertibility of $A$.

Remember, our quest is to replace the high-arity relation symbols $H_k$ from $L$ by something definable from the new signature $\tilde{L}$, such as $\hat{H_k}$ from Lemma \ref{finindsuperlemma}. Let $\hat{\alpha}$ be the $\tilde{L}$-formula obtained from $\alpha$ by replacing each instance of $H_k$ by $\psi_{H_k}$ from Lemma \ref{finindsuperlemma}.
Note that in $ \alpha(\bar{x},y) \leftrightarrow ( (\bar{x}, y)  \in A^{-1} (K + I_P) )$ only $K$ has anything to do with the $H_k$. Replacing them by the finite-index supergroups $\hat{H_k}$, we get a finite-index supergroup $\hat{K}$ of $K$, and then a finite-index supergroup $\hat{K} + I_P$ of $K +I_P$, and then a finite-index supergroup $A^{-1} (\hat{K} + I_P)$ of $A^{-1} (K + I_P)$. Thus, since there are only finitely many solutions of $\alpha(\bar{a},y)$, there are only finitely many solutions of $\hat{\alpha} (\bar{a}, y)$, yielding the desired algebraic $\tilde{L}$-formula.\end{proof}

\subsection{Kashrut and quasiendomorphisms as words}\label{kashrut}

\textbf{Setup recap.} We have a recursive strongly minimal modular group $\tilde{M}$ in a finite relational signature $\tilde{L}$ which consists of the ternary relation $+$ along with a set $S$ of binary relation symbols interpreted as quasiendomorphisms. From $\tilde{M}$, we wish to extract a recursive presentation of its quasiendomorphism ring $R$, and a recursive enumeration of its algebraic closure of the empty set, $U_0$.

Each $s \in S$ is interpreted as a rank $1$ subgroup of $\tilde{M} \times \tilde{M}$ such that the projection to the first coordinate is surjective, so $s$ is either a finite-to-finite group correspondence on $\tilde{M}$, or it is $\tilde{M} \times F$ for some finite subgroup $F \leq \tilde{M}$. These last will be called \emph{constant quasiendomorphisms}. If $s\in S$ is of the form $\tilde{M} \times F$, then (from connectedness by Lemma \ref{wlogcon}) $s$ is in fact $\tilde{M}\times \{0\}$.  To streamline notation, we assume that \textbf{one of the $s \in S$ is interpreted as the diagonal, corresponding to the identity quasiendomorphism, and another is interpreted as $\tilde{M} \times \{0\}$, corresponding to the single-valued, constant quasiendomorphism}.

We first analyze the ring of quasiendomorphisms generated (as a division ring) by $S$, and later verify that these are all quasiendomorphisms. We examine the ring generated by $S$ carefully, keeping track of which words represent constant quasiendomorphisms, so that we don't try to invert the constant ones, and so that the set $Z$ of words representing constant quasiendomorphisms turns out to be a recursive subset in the end.

\begin{deff}
 Let $W$ be the term algebra for the signature $\{ +, -, \cdot, ^{-1} \}$ on the generators $S$.

 For $w \in W$, we define $\tilde{L}$-formulae $\phi_w(x,y)$ inductively: \begin{itemize}
  \item if $w = s \in S$, let $\phi_w(x,y) := s(x,y)$
  \item if $w = v^{-1}$ for some $v \in W$, let $\phi_w(x,y) := \phi_v(y,x)$
  \item if $w = -v$ for some $v \in W$, let $\phi_w(x, y) := \phi_v(x, -y)$
  \item if $w = v + u$ for some $u,v \in W$, let $\phi_w(x,y) := \exists z_v\, \exists z_u\, \left(\phi_v(x,z_v) \wedge \phi_u(x, z_u) \wedge y=z_v+z_u\right)$
  \item if $w = v \cdot u$ for some $u,v \in W$, let $\phi_w(x,y) :=  \exists z_u\, \left(\phi_u(x, z_u) \wedge \phi_v(z_u, y)\right)$ \end{itemize}
 \end{deff}

 We will use $w$ to refer to both the word as well as the group defined by $\phi_w$. Note that each $\phi_w$ is positive primitive. We now inductively define two subsets $Z \subset K$ of $W$. Words in $K$ are called \emph{kosher} and words in $Z$ are called \emph{zero words}. Really, $K$ is just the collection of words all of whose subwords define quasiendomorphisms, and $Z$ is the set of the words in $K$ which define constant quasiendomorphisms. The purpose of giving the inductive definition is to show that both are recursive subsets of $W$.

We induct on the length of the word $w \in W$. When $w$ has length 1, $w=s$ for some $s\in S$ and $w$ is in $K$. We carry the following inductive hypotheses: \begin{enumerate}
  \item $K$ is closed under subwords, i.e.
   if $v \notin K$, then $-v$, $v^{-1}$, $v+u$, $u+v$, $v \cdot u$, and $u \cdot v$ are all not in $K$, for any $u \in W$.
  \item $K$ is closed under $+$, $-$, and $\cdot$, i.e.
   if $u, v \in K$, then $-v$, $v+u$, $u+v$, $v \cdot u$, and $u\cdot v$ are all in $K$.
  \item $\phi_w$ defines a quasiendomorphism for any $w \in K$.
  \item If $w \in K$, then either $\phi_w$ is constant and $w \in Z$ and $w^{-1} \not\in K$, or $\phi_w$ is not constant and $w \not\in Z$ and $w^{-1} \in K$.
  \end{enumerate}
Hypotheses 1,2 and 4 completely determine $K$ and $Z$, and 3 follows from the construction.

To make this recursive, we need a recursive procedure that determines whether a kosher word is a zero word.
Recall that we have $\tilde{M}$, a recursive positive-dimensional model of $\tilde{T}$, and we have $t \in \tilde{M} \smallsetminus \acl(\emptyset)$. A kosher word $w \in W$ is a zero word if and only if $(t,0) \in w$, which is $\Sigma_1$ since $\phi_w$ is defined by an existential $\tilde{L}$-formula in the recursive model $\tilde{M}$. On the other hand, a kosher word is non-zero if and only if $t$ belongs to its image, which is also defined by an existential formula in the same recursive model. Therefore, the collection of kosher words of length $n$ which are zero words is uniformly recursive, giving us a way to determine which words of length $n+1$ are kosher, demonstrating that both $K$ and $Z$ are recursive subsets of $W$. Thus, since $K$ and $Z$ are both recursive, this yields a recursive presentation $K/Z$ of the ring of quasiendomorphisms generated by $S$. We now turn our attention to verifying that all quasiendomorphisms arise in this way.


\subsubsection{Row-reduction}

To be able to recursively characterize the quasiendomorphism ring, and to recursively enumerate $\acl(\emptyset)$, we need to show that the entire quasiendomorphism ring is generated by $S$ as a division ring. In particular, if an $\tilde{L}$-formula $\theta(x,y)$ defines a quasiendomorphism $f:\tilde{M} \rightarrow \tilde{M}$, we need to find a word $w \in K$ such that $f-w$ is a constant quasiendomorphism (i.e., $f=w$ in the quasiendomorphism ring), and $w$ is a finite index supergroup of $f$. To do so, we write the quasiendomorphism $f$ as a matrix equation. Via a process of row reduction, we determine the correct word $w$. We may assume that $\theta$ is a pp formula of the form
$$\theta(x,y) = \exists z_1\, \exists z_2\, \ldots \exists z_n \bigwedge_{i \leq m}\left( (b_{i1} x + c_{i1} y + \sum_{j \leq n} (a_{ij1}z_j),b_{i2} x + c_{i2} y + \sum_{j \leq n} (a_{ij2}z_j))\in s_i \right)$$,
where $s_i\in S$ and all $a,b,c$ are in $\mathbb{Z}$. We wish to re-write this formula as a matrix equation. By a formula of the form $\sum_{i\leq n} (w_i x_i) \in F$, we mean that there are some $x_i'$ so that $(x_i,x_i')\in w_i$ and $\sum_{i\leq n} x_i' \in F$. We re-write $\theta$ as the following:

$$ \theta(x,y) = \exists z_1\, \exists z_2\, \ldots \exists z_n \bigwedge_{i \leq m}\left( (b_{i1}s_i-b_{i2}) x + (c_{i1}s_i-c_{i2}) y + \sum_{j \leq n} (a_{ij}s_i-a_{ij2})z_j) \in \im(0)\right),$$ where $0$ is the constant $0$ quasiendomorphism.
 Note that all coefficients of any $x,y,$ or $z_j$ are words in $K$. Also, $\im(0)$ is a finite definable group.

We work with augmented matrices of the form $N=(\vec{b},\vec{c},\vec{a_1},\ldots, \vec{a_n} \vert \vec{w})$, where $b_i, c_i, a_{ij} \in K$ and $w_i$ are elements of $Z$. Associated to this matrix is the group $$G_N:= \{(x,y,\bar{z})\vert (\vec{b},\vec{c},\vec{a_1},\ldots, \vec{a_n}) \cdot (x,y,\bar{z})\in \im(\vec{w})\}=$$ $$=\{(x,y,\bar{z})\vert \bigwedge_{i\leq m}( b_i x + c_i y + \sum_{j \leq n} (a_{ij}z_j) \in \im(w_i))\}$$ The group we are really interested in is $H_N$, the projection of $G_N$ upon the first 2 coordinates. Throughout the row-reduction, we maintain the \textbf{inductive hypothesis: $H_N$ is a finite index supergroup of the quasiendomorphism $f$ with which we start.}

Here are the operations we perform on these augmented matrices while maintaining the inductive hypothesis.

 If some coefficient, say $a_{i1}$, is in $Z$, its image is a finite group and its output does not depend on its input, so the equation $b_i x + c_i y + \sum_{j \leq n} (a_{ij}z_j) \in \im(w_i)$ is exactly equivalent to $b_i x + c_i y + \sum_{2 \leq j \leq n} (a_{ij}z_j) \in \im(w_i+a_{i1})$. This works equally well for any $b_i$, $c_i$, or $a_{ij}$. Doing this as much as necessary in-between all other operations, we may assume that each $a_{ij} \in Z$ if and only if $a_{ij} = 0$, and similarly for $b_i$ and $c_i$.

 For any $w \in K$, we can add a $w$-multiple of one row to another row. For example, the formula
$$b_1 x + c_1 y + \sum_{j \leq n} (a_{1j}z_j) \in \im(w_1)\,  \wedge \, b_2 x + c_2 y + \sum_{j \leq n} (a_{2j}z_j) \in \im(w_2)$$
 implies
 $$ b_1 x + c_1 y + \sum_{j \leq n} (a_{1j} z_j) \in \im(w_1)\, \wedge \,
(b_2 +w b_1) x + (c_2 + w c_1) y + \sum_{j \leq n} ( (a_{2j} + w a_{1j})z_j ) \in \im(w_2 + w w_1)$$
which in turn implies
$$ b_1 x + c_1 y + \sum_{j \leq n} (a_{1j} z_j) \in \im(w_1) \, \wedge $$ $$\wedge \,
(b_2 +w b_1 -w b_1) x + (c_2 + w c_1 -w c_1) y + \sum_{j \leq n} ( (a_{2j} + w a_{1j}- w a_{1j})z_j ) \in \im(w_2 + ww_1-ww_1)$$
which is equivalent to
$$b_1 x + c_1 y + \sum_{j \leq n} (a_{1j}z_j) \in \im(w_1)\, \wedge $$ $$ \wedge \, b_2 x + c_2 y + \sum_{j \leq n} (a_{2j}z_j) \in \im(w_2+ww_1-ww_1+wb_1-wb_1+wc_1-wc_1+ \sum_j w a_{1j} -w a_{ij})
$$

The word $w_2+ww_1-ww_1+wb_1-wb_1+wc_1-wc_1+ \sum_j w a_{1j} -w a_{ij}$ is clearly also in $Z$ and its image contains $\im(w_1)$. So, if we add a multiple of one row in our augmented matrix $N$ to another row to form $N'$, the new subgroup $H_{N'}$ will be a finite-index extension of $H_N$, which maintains our inductive hypothesis. Note that a similar analysis yields that we can multiply a row by a word $v\in K\smallsetminus Z$ without violating our inductive hypothesis. This uses the fact that $v^{-1}vw-w$ is in $Z$ for any $w\in K$ and $v\in K\smallsetminus Z$.

This allows us to row-reduce as in linear algebra. We can eliminate non-zero but constant coefficients by moving them into the image, and we can eliminate unwanted non-constant coefficients via adding a multiple of one row to another. We treat the first two columns separately. Since $\theta(x,y)$ defines a quasiendomorphism, at least one of the coefficients of $y$ must be nonconstant; we may assume that this is $c_1$. We can multiply the first row by $c_1^{-1}$, so that $c_1=1$ and then subtract a $c_i$-multiple of the first row from the $i$th row for each $i \geq 2$. If the coefficient of $x$ in some equation other than the first is now non-constant, we assume this is $b_2$ and similarly use the second row to obtain a new augmented matrix where $b_2=c_1=1$, $b_1 =c_2 =0$ and $b_i = c_i =0$ for $i \geq 3$; this is the general case. Otherwise, we obtain an augmented matrix where $b_i = c_i =0$ for $i \geq 2$ and $c_1=1$; this is the degenerate case.

Now we leave the top alone (two rows in the general case, one in the degenerate case), and row-reduce $\{ a_{ij} \}_{2\text{ or 3} \leq i \leq m, j \leq n}$ in the usual way, to obtain, possibly after reordering $\bar{z}$ a matrix $A = (I \hat{A})$ where $I$ is an identity matrix. (As usual, we drop any rows where all coefficients are zero.) Now we use these to get rid of the coefficients above the $I$ part in the top (one or two) rows, after which all but those (one or two) rows becomes irrelevant.

So we have row-reduced $\theta$ to
$$\exists \bar{z} \,\left( b x + y + \sum_{j \leq n} (a_{j}z_j) \in \im(w)\right)$$
 in the degenerate case, or to
$$\exists \bar{z} \left( \, y + \sum_{j \leq n} (a_{1j}z_j) \in \im(w_1) \, \wedge\, x + \sum_{j \leq n} (a_{2j}z_j) \in \im(w_2) \right)$$
in the general case.

In the degenerate case, this only defines a quasiendomorphism if $a_{j} =0$ for all $j$, in which case, $H_N= w-b $. In the general case, this only defines a quasiendomorphism if there is some $v\in K$ so that $a_{1j} = v a_{2j}$, up to zero words.
Finally, adding $-v$ times the second row to the first, we get the equation $y-vx\in \im (w_1-vw_2 + w_3)$, where $w_3=\sum_{j\leq n} (a_{1j} - v a_{2j})$. Then the word we need is $v+w_1-vw_2+w_3$.

 Thus we have shown that all quasiendomorphisms are finite index subgroups of quasiendomorphisms defined by words.

\begin{cory}
\begin{itemize}
\item The quasiendomorphism ring has a recursive presentation.
\item $\acl(\emptyset)$ is a $\Sigma_1$ subset of $U$.
\item The prime model of $T$ has a recursive presentation.
\end{itemize}
\end{cory}
\begin{proof}
We have shown that every quasiendomorphism is equivalent to a member of $K$. Thus the quasiendomorphism ring is equal to $K/Z$. Since each of $K$ and $Z$ are recursive, this gives a recursive presentation.

If $b\in \acl(\emptyset)$, then by the form of the algebraic formula from (\ref{reducedformalg}) on p. \pageref{reducedformalg} with $\bar{x}=\emptyset$, we see that $b$ is in a finite $\emptyset$-definable group $F$. Thus $M \times F$ is a quasiendomorphism. Thus there is a word $z\in Z$ such that $M \times F \leq z$. Thus we see that $\acl(\emptyset)=\bigcup_{z\in Z} \im(z)$, which is naturally a $\Sigma_1$ set.

If $\acl(\emptyset)$ is infinite, then the previous claim gives a recursive presentation of the prime model. Suppose that $\acl(\emptyset)=F$ is a finite subgroup of $M$. We show that $\acl(a)$ is $\Sigma_1$ for a generic $a$. By Fact 3.1 of \cite{BnB}, $b\in \acl(a)$ if and only if there is some $w\in K\smallsetminus Z$ and some $d\in \acl(\emptyset)$ so that $(a,b+d)\in w$. This is again naturally a $\Sigma_1$ set.
\end{proof}

\subsection{Triumphant}\label{triumph}

We have now analyzed the quasiendomorphism ring of $\tilde{T}$, which has the same algebraic closure relation as that of $T$. Following \cite{BnB}, we define $T_1:=\Th(M/\acl(\emptyset))$ and $\tilde{T}_1:=\Th(\tilde{M}/\acl(\emptyset))$. We first note that $T_1$ and $\tilde{T}_1$ are $\Delta_1$-interdefinable. This follows immediately from Lemma \ref{finindsuperlemma} and the axioms for $T_1$ and $\tilde{T}_1$ (found on page 31 of \cite{BnB}), which state that every pair of definable groups $G\leq H$ with finite index has index 1. Lemma \ref{finindsuperlemma} says that for every $G\in L$ there is a finite index supergroup $\hat{G}$ of $G$ which is $\Delta_1$-definable in $\tilde{T}$. From the axiomatization, we see that $T_1\models G=\hat{G}$.

From our recursive presentation $K/Z$ of the quasiendomorphism ring $R$, we recursively present every model of $\tilde{T}_1$. The axioms of $\tilde{T}_1$ say that a model is an $R$-vector space. An $R$-vector space of dimension $d\in \omega+1$ can be presented recursively as $(K/Z)^d$. This gives a recursive presentation of every model of $\tilde{T}_1$, and, by the previous paragraph, of every model of $T_1$.

By proposition 2.13 of \cite{BnB}, every direct sum of a model of $T$ and a model of $T_1$ is again a model of $T$. Therefore, letting $M_0$ be a recursive prime model of $T$, and letting $N$ be a recursive $d$-dimensional model of $T_1$, we obtain a recursive model $M_0\oplus N$ of $T$. By Fact 3.1 from \cite{BnB} characterizing algebraicity, the dimension over the prime model of this model is $d$. Thus for an arbitrary $d\in \omega+1$, we have given a recursive presentation of the model of $T$ of dimension $d$ over the prime model. Thus we have proved our main theorem:

\begin{thm}\label{modularspec}
If $T$ is a modular strongly minimal theory in a finite signature expanding a group, then $\Spec(T)= \emptyset, \omega+1,\text{ or }\{0\}$.
\end{thm}
\begin{proof}
We have shown that given a recursive positive-dimensional model of $T$, there is a recursive presentation of every countable model of $T$. These are the only spectra consistent with this.
\end{proof}

\bibliographystyle{amsplain}

\end{document}